\documentclass[CRMATH,Unicode,manuscript]{cedram}

\usepackage{xcolor}
\usepackage{cleveref}

\newcommand{\ext}{\operatorname{ext}}
\newcommand{\Lip}{\operatorname{Lip}}
\newcommand{\conv}{\text{conv}}

\title{On extreme points and representer theorems for the Lipschitz unit ball on finite metric spaces}

\author{\firstname{Kristian} \lastname{Bredies}\CDRorcid{0000-0001-7140-043X}}
\address{Institute of Mathematics and Scientific   Computing, University of Graz, Graz, Austria}
\email[K. Bredies]{kristian.bredies@uni-graz.at}
\author{\firstname{Jonathan} \lastname{Chirinos Rodriguez*}\CDRorcid{0000-0003-4107-5108}}
\address{MaLGa, DIMA, Università degli Studi di Genova and Camelot Biomedical Systems, Genova, Italia}
\email[J. Chirinos Rodriguez]{jonathan.chirinosrodriguez@edu.unige.it}

\author{\firstname{Emanuele} \lastname{Naldi}\CDRorcid{0000-0002-2437-2832}}
\address{Institute  of  Analysis  and  Algebra,  TU  Braunschweig,  Braunschweig,  Germany}
\email[E. Naldi]{e.naldi@tu-braunschweig.de}

\begin{abstract}
	In this note, we provide a characterization for the set of extreme points of the Lipschitz unit ball in a specific vectorial setting. While the analysis of the case of real-valued functions is covered extensively in the literature, no information about the vectorial case has been provided up to date. Here, we aim at partially filling this gap by considering functions mapping from a finite metric space to a strictly convex Banach space that satisfy the Lipschitz condition. As a consequence, we present a representer theorem for such functions. In this setting, the number of extreme points needed to express any point inside the ball is independent of the dimension, improving the classical result from Carathéodory.
\end{abstract}

\begin{document}

\maketitle

\section{Introduction}
Let $(\mathcal{X}, d)$ be a metric space, $(\mathcal{Y}, \|\cdot\|)$ a non-trivial, strictly convex real Banach space and consider the Banach space of Lipschitz functions \cite{Weaver} from $\mathcal{X}$ to $\mathcal{Y}$ that vanish at the distinct point $x_0\in \mathcal{X}$. Let $L\geq 0$. We define the set
$$
\Lip^L_0 =\left\{f:\mathcal{X}\to\mathcal{Y}  \mid  f(x_0)=0 \text{ and } \|f(x)-f(y)\|\leq Ld(x, y) \text{ for all } x, y\in\mathcal{X} \right\}.
$$
In this paper, we want to study certain structural properties of the special case $\Lip_0^1$ (since all of the other cases can be treated analogously). In particular, we are interested in characterizing the set of its extreme points. The set of extreme points of the space $\Lip^1_0$ above mentioned has been widely studied in the past decades for real-valued functions $f:\mathcal{X}\to \mathbb{R}$ (see \cite{cobz,Farmer94,raoashoke,rolew,ashoke,smarz}), and more recently in \cite{alipazan}, but no information about the set $\ext(\Lip^1_0)$ has been provided in the case of $\mathcal{Y}\neq \mathbb{R}$. In this paper, we show that the latter space can be characterized when considering a finite metric space $\mathcal{X}=\{x_0,...,x_n\}$, being $x_0,\dots, \, x_n$, $n\geq 1$, distinct points.

The name \textit{representer theorems} was introduced in the field of Machine Learning. In particular, in the context of kernel methods \cite{schsmo}. In a few words, these results show that any element of a space can be expressed by a linear combination of finitely many specific points. Recently, the study of representation results has gained popularity in the setting of variational inverse problems \cite{boycham,breca2020,unser2020}. Moreover, and motivated by the so-called Minkowski--Carathéodory theorem \cite[Theorem III.2.3.4]{hiur}, which yields representations in terms of extreme points in the finite-dimensional setting, it has been observed that there is a natural connection between extreme points and representation results \cite{duval}. As shown in \cite{brelinear,breetal22} a proper characterization of extreme points may lead to efficient optimization algorithms. For these reasons, there is increasing recent interest in characterizing extreme points associated with various regularizers, see \cite{breca2020,breca2021} and \cite{brelinear} and \cite{unfawa17,brecafa22,carioni2022extremal,iglesias2022extreme,ambrosio2022linear}. Finally, we mention that Lipschitz-type constraints have also recently become of interest in the context of plug-and-play regularization \cite{ryupnp2019} and monotone splitting algorithms \cite{pescterr2021}. Hence, a proper characterization of the extreme points of the Lipschitz unit ball may have a considerable impact. In this paper, we aim to partially fill this gap.

We present the characterization result in Theorem \ref{thm:extrpo}. Finally, we provide in Theorem \ref{thm:repr} a representer theorem for the space $\Lip_0^1 $ that improves the Minkowski--Carathéodory theorem, in the sense that we will see that the number of required extreme points is independent of the dimension of the space.

\section{Extreme points and representer theorems}
We recall that, given $C$ a convex set of a real vector space, an \emph{extreme point} of $C$ is a point $y\in C$ such that, if $y=\lambda y^1 + (1-\lambda)y^2$ with $y^1, \ y^2\in C$ and $\lambda \in (0,1)$, then $y^1=y^2=y$. In other words, an extreme point of a convex set $C$ is a point $y$ in $C$ such that $C\setminus \{y\}$ is convex. We remind the reader of the well-known fact that it is sufficient to check the extreme point condition only for $\lambda=1/2$. We denote with $\ext(C)$ the set of extreme points of $C$.

We consider, for further convenience, the following definition of the set $\Lip_0^1$.
$$
	\Lip^1_0 :=\left\{y=(y_0,\dots, \, y_n)\in\mathcal{Y}^{n+1} \, : \, y_0=0 \text{ and }  \|y_i-y_j\|\leq d(x_i,x_j), \text{ for every } i,j=1,..., \ n\right\}.
$$
Note that both of the definitions of the set $\Lip_0^1 $ are equivalent, since in this case, we are only considering the images of the finite set $\mathcal{X}=\{x_0,...,x_n\}$ through functions $f$ mapping to $\mathcal{Y}$. We provide now a preliminary lemma.

\begin{lemma}\label{lemma} Let $y\in\Lip^1_0 $. For every $i=1,..., \, n$, there exists $0=i_0, \, i_1,.., \, i_k=i$, $k\geq 1$, such that $\|y_{i_{j+1}}-y_{i_j}\|=d(x_{i_{j+1}},x_{i_j})$, for every $j=0,..., \, k-1$ if and only if there does not exist a nonempty subset $S\subset\{1,..., \,n\}$ such that $\|y_i-y_j\|<d(x_i,x_j)$, for every $i\in S$, $j\in S^c$.
\end{lemma}
\begin{proof}
	First, we proceed by contradiction: let $S\subset\{1,..., \, n\}$, $S\neq \emptyset$, such that $\|y_i-y_j\|<d(x_i,x_j)$ for every $i\in S$, $j\in S^c$, and let $i\in S$. By hypothesis, we can choose $k\geq 1$ with $0=i_0,..., \, i_k=i$ such that $\|y_{i_{j+1}}-y_{i_j}\|=d(x_{i_j}, x_{i_{j+1}})$ for every $j=0,...,\, k-1$. As $i_0=0\in S^c$ and $i_k=i\in S$, we derive that there must exist $j=0,..., \, k-1$ such that $i_{j+1}\in S$ and $i_j\in S^c$. It follows that $\|y_{i_{j+1}}-y_{i_j}\|<d(x_{i_{j+1}}, x_{i_{j}})$ but this contradicts the hypothesis and, hence, concludes the first part of the proof.\\ Conversely, let us consider the set
	\begin{align*}
		T=\left\{i\in\{1,..., \, n\} \, \mid \ \text{there exists } \ 0=i_0,..., \, i_k=i \, \text{ s.t. }
		\|y_{i_{\ell+1}}-y_{i_{\ell}}\|=d(x_{i_{\ell+1}},x_{i_{\ell+1}}), \, \ell=0,..., \, k-1\right\},
	\end{align*}
	and suppose that $T\neq \{1,..., \, n\}$. Define the set $S:=T^c\subset\{1,..., \, n\}$, and observe that $S\neq \emptyset$. It is left to prove that, for every $i\in S$, $j\in S^c$, $\|y_i-y_j\|< d(x_i,x_j)$. Let us suppose that there exists $i\in S$ and $j\in S^c$ such that $\|y_i-y_j\|=d(x_i,x_j)$. Since $j\in S^c$, there exists $0=i_0,..., \, i_k=j$, $k\geq 1$, such that $\|y_{i_{\ell+1}}-y_{i_{\ell}}\|=d(x_{i_{\ell}},x_{i_{\ell+1}})$ for $\ell=0,..., \, k-1$. Since, by hypothesis, we have that $\|y_i-y_j\|=d(x_i,x_j)$, and defining $i_{k+1}:=i$ we obtain a path from $0$ to $i$ satisfying the equalities for $\ell=1,...,\, k$. This implies that $i\in T=S^c$, a contradiction. We therefore have found that there exists $S\subset \{1,..., \, n\}$, $S\neq\emptyset$ such that $\|y_i-y_j\|<d(x_i,x_j)$, for every $i\in S$, $j\in S^c$.
\end{proof}
Define now the set
$$
	\begin{aligned}
		\mathcal{E}:=\displaystyle\{y\in\mathcal{Y}^{n+1} \, \mid \, y_0=0 \text{ and, for every } i & =1,..., \, n, \text{ there exists } 0=i_0,..., i_k=i, \ k\geq 1 :                  \\
		                                                                                             & \displaystyle\|y_{i_{j+1}}-y_{i_j}\|= d(x_{i_j}, x_{i_{j+1}}), \ j=0,..., \ k-1\}.
	\end{aligned}
$$
Observe that the definition of $\mathcal{E}$ is motivated by the previous lemma, since every point $y\in\mathcal{E}$ satisfies the first condition of Lemma \ref{lemma}. We are now ready to characterize the extreme points of the set $\Lip_0^1 $.
\begin{theo}\label{thm:extrpo}
	We have that $\ext (\Lip_0^1 )=\mathcal{E}$.
\end{theo}
\begin{proof}
	First, we will prove that, if $y\notin \mathcal{E}$, then $y\notin \ext(\Lip_0^1 )$. Let $y\in\Lip_0^1 $ such that $y\notin \mathcal{E}$. By the previous lemma, we get that there exists $S\subset\{1,..., \, n\}$, $S\neq \emptyset$, such that $\|y_i-y_j\|< d(x_i,x_j)$, for every $i\in S$, $j\in S^c$. Choose now
	$$
		\varepsilon=\min_{i\in S, \, j\in S^c} d(x_i,x_j)-\|y_i-y_j\|,
	$$
	and observe that $\varepsilon>0$. Moreover, choose $v\in \mathcal{Y}$ such that $\|v\|=1$ (which exists since $\mathcal{Y}$ is non-trivial) and set
	$$
		y_i^1:=
		\begin{cases}
			y_i+\varepsilon v, & \text{ if } i\in S; \\
			y_i,               & \text{else},
		\end{cases}
		\quad
		y_i^2:=
		\begin{cases}
			y_i-\varepsilon v, & \text{ if } i\in S; \\
			y_i,               & \text{else}.
		\end{cases}
	$$
	Indeed, if we define $y^k:=(y^k_0, y^k_1,..., \, y^k_n)$, $k=1, \, 2$, then $y^1\neq y^2$. Moreover, observe that
	$$
		\|y^k_i-y^k_j\|=\|y_i-y_j\|\leq d(x_i,x_j), \quad \text{for every } i, \, j\in S \text{ or } i, \, j\in S^c, \, k=1, \, 2,
	$$
	since $y\in\Lip_0^1$ and
	\begin{align*}
		\|y^k_i-y^k_j\| & =\|y_i \pm\varepsilon v-y_j\|\leq\|y_i-y_j\|+\varepsilon \leq \|y_i-y_j\|+d(x_i,x_j)-\|y_i-y_j\| \\
		                & =d(x_i,x_j),\quad \text{for } i\in S, \, j\in S^c, \, k=1, \, 2.
	\end{align*}
	Therefore, $y^k\in\Lip_0^1$, $k=1, \, 2$ and $y=\frac 1 2 y^1+\frac 1 2 y^2$, $y^1\neq y^2$. Hence, $y\notin \ext(\Lip_0^1)$ and so, $\ext(\Lip_0^1)\subset \mathcal{E}$. \\
	We would like to prove now that $\mathcal{E}\subset \ext(\Lip_0^1 )$. Let $y\in\Lip_0^1 \setminus\ext(\Lip_0^1 )$. We will prove that there exists $S\subset\{1,..., \, n\}$, $S\neq \emptyset$, such that $\|y_i-y_j\|<d(x_i,x_j)$, for every $i\in S$, $j\in S^c$. If so, by the previous lemma, this would mean that $y\notin \mathcal{E}$. Since $y\notin \ext(\Lip_0^1 )$, there exist $y^1$, $y^2\in \Lip_0^1$, $y^1\neq y^2$, such that $y=\frac 1 2 y^1 +\frac 1 2 y^2$. Now, define the set $S=\{i\in\{1,..., \, n\} \, \mid \, y_i^1\neq y_i^2\}$ and observe that it is nonempty since $y^1\neq y^2$ by hypothesis. Now, let $i\in S$, $j\in S^c$. Then,
	$$
		\left\|y_i-y_j\right\|=\left\|\frac 1 2 y_i^1 -\frac 1 2 y_j^1+\frac 1 2 y_i^2-\frac 1 2 y_j^2\right\|=\left\|\frac 1 2 y_i^1 -\frac 1 2 y_j+\frac 1 2 y_i^2-\frac 1 2 y_j\right\|.
	$$
	In order to finish the proof, define  $a:= y_i^1 - y_j$, $b:=y_i^2- y_j$, and observe that $a\neq b$. Now, we distinguish two cases: if $a$ is not proportional to $b$, we get
	$$
		\left\|y_i-y_j\right\|< \frac 1 2 \|a\| + \frac 1 2 \|b\|\leq
		d(x_i,x_j),
	$$
	since we assumed that $\mathcal{Y}$ is a strictly convex space. If they are proportional, then, by possibly interchanging $a$ and $b$, we have $b=\lambda a$ for some $\lambda\neq 1$, we can further assume that $-1\leq \lambda< 1$, and obtain that
	$$
		\left\|y_i-y_j\right\|=\left\|\frac a 2 +\frac{\lambda a} {2}\right\|\leq \frac{|1+\lambda|}{2}\|a\|<d(x_i,x_j).
	$$
	The result immediately follows.
\end{proof}

We are now ready to state the representer theorem for the space $\Lip_0^1$. In the case of $\mathcal{Y}=\mathbb{R}^d$, the Minkowski--Carathéodory theorem would imply that every function in $\Lip_0^1 $ can be represented as a convex combination of at most $nd+1$ points. We are able to improve this number up to $n+1$ points, which is independent of $d$, and covers the infinite-dimensional case as well.

\begin{theo}\label{thm:repr}
	For every $y\in\Lip_0^1 $, there exist $k\leq n+1$, $y^1,..., \, y^k\in\ext(\Lip_0^1 )$, and scalars $\lambda_1,..., \lambda_k\geq 0$ with $\sum_{i=1}^k\lambda_i=1$ such that $y=\sum_{i=1}^k\lambda_iy^i$.
\end{theo}
\begin{proof}
	Let $y\in\Lip_0^1 $ and choose $v\in\mathcal{Y}$ such that $\|v\|=1$. Define the set
	$$
		D=\left\{t=(t_0,..., \, t_n)\in\mathbb{R}^{n+1} \, \mid \, y+tv\in \Lip_0^1\right\},
	$$
	being $(y+tv)_i:=y_i+t_iv$, for every $i=0,..., \, n$. Moreover, observe that $t_0=0$ for every $t\in D$ since, if $t_0\neq 0$ then  $(y+tv)_0\neq 0$. Now, we claim that, if $t\in \ext(D)$, then $y+tv\in\ext(\Lip_0^1 )$. Indeed, if $t\in D$ and $y+tv\notin\ext(\Lip_0^1)$, then there exists a subset $S\subset\{1,..., \, n\}$, $S\neq\emptyset$, such that $\|y_i-y_j+(t_i-t_j)v\|<d(x_i,x_j)$, for every $i\in S$, $j\in S^c$. Choose
	$$
		\varepsilon=\min_{i\in S, \, j\in S^c} d(x_i,x_j)-\|y_i-y_j+(t_i-t_j)v\|,
	$$
	and observe that $\varepsilon>0$. Moreover, define
	$$
		t_i^1:=
		\begin{cases}
			t_i+\varepsilon, & \text{ if } i\in S; \\
			t_i,             & \text{else},
		\end{cases}
		\quad
		t_i^2:=
		\begin{cases}
			t_i-\varepsilon, & \text{ if } i\in S; \\
			t_i,             & \text{else}.
		\end{cases}
	$$
	With such definitions, observe that $t^1\neq t^2$. Now, $y+t^kv\in \Lip_0^1$, for $k=1, \, 2$, because
	$$
		\|y_i-y_j+(t^k_i-t^k_j)v\|=\|y_i-y_j+(t_i-t_j)v\|\leq d(x_i,x_j), \quad \text{for every } i, \, j\in S \text{ or } i, \, j\in S^c,
	$$
	since $t\in D$ and
	\begin{align*}
		\|y_i-y_j+(t^k_i-t^k_j)v\| & \leq \|y_i-y_j+(t_i-t_j)v\|+\varepsilon\|v\|                      \\
		                           & \leq \|y_i-y_j+(t_i-t_j)v\|+d(x_i,x_j) - \|y_i-y_j+(t_i-t_j)v\|   \\
		                           & =d(x_i,x_j), \quad \text{for } i\in S, \, j\in S^c, \, k=1, \, 2.
	\end{align*}
	Then, $t^1$, $t^2\in D$ and $t=\frac 1 2 t^1+\frac 1 2 t^2$, $t^1\neq t^2$, which implies that $t\notin\ext(D)$. Consequently, $t\in\ext(D)$ implies $y+tv\in\ext(\Lip_0^1 )$. Now, we show that $D$ is a nonempty, convex, compact subset of $\mathbb{R}^{n+1}$. First, note that $0\in D$ and that convexity follows from fact that $D$ is the preimage of the convex set $\Lip_0^1 $ through the affine mapping $t\mapsto y+tv$. Moreover, boundedness follows because, for every $t\in D$, we have
	\[d(x_i,x_0)\geq \|(y+tv)_i-(y+tv)_0\|=\|y_i-y_0+t_iv\|\geq |t_i|-\|y_i-y_0\| \]
	and so, for every $i=1,\dots,n$ we have that
	$$
		|t_i|\leq d(x_i,x_0) + \|y_i-y_0\| \leq 2 d(x_i,x_0).
	$$
	It is only left to prove that $D$ is closed. Let $(t^k)_{k\in\mathbb{N}}$ be a sequence in $D$ converging to some $t\in\mathbb{R}^{n+1}$. We have that
	\begin{align*}
		\|y_i-y_j+(t_i-t_j)v\| & = \|y_i-y_j+(t_i^k-t_j^k)v-(t_i^k-t_j^k)v+(t_i-t_j)v\|                           \\
		                       & \leq \|y_i-y_j+(t_i^k-t_j^k)v\|+ \|(t_i-t_j)v-(t_i^k-t_j^k)v\|                   \\
		                       & \leq d(x_i,x_j)+|t_i-t_i^k|+|t_j-t_j^k|, \quad \text{ for every } i,j=0,\dots,n.
	\end{align*}
	We obtain the result by taking limits when $k\to \infty$. By the Krein--Milman theorem, we know that $D=\overline{\conv}(\ext(D))$. Moreover, we can apply the Minkowski--Carathéodory theorem, and since  $0\in D$, $\mathrm{span} \ D\subset \{0\}\times\mathbb{R}^n$, we have $\mathrm{dim} \ \mathrm{span} \ D\leq n$. Consequently, there exist $k\leq n+1$ and scalars $\lambda_1,..., \, \lambda_k\geq 0$ with $\sum_{i=1}^k\lambda_i=1$ such that $0=\sum_{i=1}^k\lambda_it^i$, with $t^i\in\ext(D)$, $i=1,..., \, k$. Finally, by the previous claim, we know that for every $i=1,..., \, k$, if we define $y^i:= y+t^iv$, then $y^i\in\ext(\Lip_0^1 )$ and, hence
	$$
		\sum_{i=1}^k \lambda_iy^i=\sum_{i=1}^k \lambda_i (y+t^i v)=y+\left(\sum_{i=1}^k\lambda_it^i\right)v=y,
	$$
	concluding the proof.
\end{proof}

\section{Acknowledgements}
The authors would like to thank Rodolfo Assereto, Enis Chenchene for the productive discussions during the secondment period of both J.C.R. and E.N. at the University of Graz. This project has been supported by the TraDE-OPT project which received funding from the European Union’s Horizon 2020 research and innovation program under the Marie Skłodowska-Curie grant agreement No 861137.  The Institute of Mathematics and Scientific Computing at the University of Graz, with which K.B. is affiliated, is a member of NAWI Graz (\url{https://nawigraz.at/en}). This work represents only the view of the authors. The European Commission and other organizations are not responsible for any use that may be made of the information it contains.

\bibliographystyle{crplain}

\bibliography{samplebib}

\end{document}